\documentclass[11pt,final]{amsart}
\usepackage{a4wide}

\usepackage{amsmath,amsthm,amssymb}

\usepackage{enumitem}

\usepackage{float}
\usepackage{ifthen}
\usepackage{xargs}
\usepackage{url}
\usepackage[all]{xy}
\usepackage[rgb,dvipsnames]{xcolor}
\usepackage{setspace}
\usepackage[latin1]{inputenc}
\usepackage[pdfborder={0 0 0}, colorlinks = true, linkcolor=blue, citecolor=OliveGreen, pdftex]{hyperref}
\usepackage[capitalize]{cleveref}

\newcommand{\R}{\mathbb{R}} 
\newcommand{\Q}{\mathbb{Q}} 
\newcommand{\uhp}{\mathbb{H}}

\newcommand{\C}{\mathbb{C}} 
 
\newcommand{\Z}{\mathbb{Z}} 
\newcommand{\F}{\mathbb{F}} 
 
\newcommand{\N}{\textup{N}}

\newcommand{\Zhat}{\hat{\mathbb{Z}}}
\newcommand{\adeles}{\mathbb{A}}

\newcommand*{\domain}{\mathbb{D}}

\newcommand*{\reg}{\mathrm{reg}}

\newcommand{\OD}{\calO_{D}}

\newcommand{\ODhatt}{\hat{\calO}_{D}^{\times}}

\newcommand{\bs}{\backslash}

\newcommand{\calB}{\mathcal{B}}
\newcommand{\calC}{\mathcal{C}}

\newcommand{\calF}{\mathcal{F}}

\newcommand{\calI}{\mathcal{I}}

\newcommand{\calO}{\mathcal{O}}

\newcommand{\calS}{\mathcal{S}}

\newcommand{\different}[1]{\mathfrak{d}_#1}

\newcommand{\frakA}{\mathfrak{A}}
\newcommand{\fraka}{\mathfrak{a}}
\newcommand{\frakb}{\mathfrak{b}}

\newcommand{\frakd}{\mathfrak{d}}
\newcommand{\frake}{\mathfrak{e}}

\newcommand{\frakp}{\mathfrak{p}}

\newcommand{\legendre}[2]{\left( \frac{\mathstrut #1}{#2} \right)}

\newcommand{\Tmatrix}{\begin{pmatrix}1 & 1 \\ 0 & 1\end{pmatrix}}
\newcommand{\Smatrix}{\begin{pmatrix}0 & -1 \\ 1 & 0\end{pmatrix}}

\newcommand{\twomat}[4]{\begin{pmatrix}
                        #1 & #2 \\ #3 & #4 
                       \end{pmatrix}}

\newcommandx*{\modformsspace}[4][1=,2=,3=]{#4_{\ifthenelse{\equal{#1}{}}{}{#1
				\ifthenelse{\equal{#2}{}}{}{,\wrep[#2][#3]{}}
				}}}
\newcommandx*{\modf}[3][1=,2=,3=]{\modformsspace[#1][#2][#3]{M}}
\newcommandx*{\wmodf}[3][1=,2=,3=]{\modformsspace[#1][#2][#3]{M^!}}

\newcommandx*{\hwmf}[3][1=,2=,3=]{\modformsspace[#1][#2][#3]{\mathcal{N}}}
\newcommandx*{\hwmfpole}[3][1=,2=,3=]{\modformsspace[#1][#2][#3]{N}}
\newcommandx*{\cuspf}[3][1=,2=,3=]{\modformsspace[#1][#2][#3]{S}}

\newcommandx*{\wrep}[3][1=,2=]{
	\ifthenelse{\equal{#1}{}}{\omega}{
		\ifthenelse{\equal{#2}{}}{\rho}{\overline{\rho}}_{#1}}
	\ifthenelse {\equal {#3}{}} { } { \left(#3\right) } }
\newcommand{\abs}[1]{\left\vert#1\right\vert}

\DeclareMathOperator{\Aut}{Aut}

\DeclareMathOperator{\SL}{SL}
\DeclareMathOperator{\GL}{GL}

\DeclareMathOperator{\Cl}{Cl}
\DeclareMathOperator{\Clk}{Cl_{k}}
\DeclareMathOperator{\Clkd}{Cl_{k}^{\ast}}
\DeclareMathOperator{\Clks}{Cl_{k}^{2}}

\DeclareMathOperator{\tr}{tr}

\DeclareMathOperator{\Og}{O}
\DeclareMathOperator{\SO}{SO}

\DeclareMathOperator{\GSpin}{GSpin}

\DeclareMathOperator{\sgn}{sgn}

\DeclareMathOperator*{\CT}{CT}

\DeclareMathOperator{\res}{res}

\newtheorem{theorem}{Theorem}[section]
\newtheorem{proposition}[theorem]{Proposition}
\newtheorem{lemma}[theorem]{Lemma}
\newtheorem{corollary}[theorem]{Corollary}

\theoremstyle{definition}
\newtheorem{definition}[theorem]{Definition}

\newtheorem{remark}[theorem]{Remark}

\newcommand{\QD}[1][]{Q_\Delta\ifthenelse{\equal{#1}{}}{}{\left(#1\right)}}

\renewcommand{\S}{\mathcal{S}}

\newcommand{\new}{\mathrm{new}}

\newcommand{\colvec}[2]{\begin{pmatrix}  #1 \\ #2 \end{pmatrix}}

\DeclareMathOperator{\sym}{sym}

\hypersetup
{
    pdfauthor={Stephan Ehlen},
    pdftitle={Vector valued theta functions associated with binary quadratic forms},
    pdfkeywords={theta functions, binary quadratic forms, CM values}
}
\numberwithin{equation}{section}

\begin{document}
\title[Vector valued theta functions associated with binary quadratic forms]{Vector valued theta functions associated with \\ binary quadratic forms}
 
\author{Stephan Ehlen}
\email{stephan.ehlen@mcgill.ca}
\address{McGill University, Department of Mathematics and Statistics, 805 Sherbrooke St. West, Montreal, Quebec, Canada H3A 0B9}

\begin{abstract}
We study the space of vector valued theta functions for the Weil representation of a positive definite even lattice of rank two
with fundamental discriminant.
We work out the relation of this space to the corresponding scalar valued theta functions of weight one
and determine an orthogonal basis with respect to the Petersson inner product.
Moreover, we give an explicit formula for the Petersson norms of the elements of this basis.
\end{abstract}

\thanks{This work was partly supported by DFG grant BR-2163/2-1.}

\maketitle

\section{Introduction and statement of results}
Integral binary quadratic forms and the automorphic properties of their theta functions 
are well known.  It is the purpose of the present note to describe the related space of vector valued theta functions
transforming with the Weil representation.

Let $P$ be an even positive-definite lattice of rank $2$ with quadratic form $Q$.
For simplicity, we assume that the discriminant $D<0$ of $Q$ is a fundamental discriminant.
The theta function attached to $P$ is a holomorphic modular form of weight $1$
and transforms with the Weil representation $\rho_P$ of $\SL_2(\Z)$ (see \cref{sec:regul-theta-lifts}).
In fact, there is a family of theta functions attached to $P$ that have the same weight
and transformation behaviour. These theta functions essentially correspond to the lattices in the genus of $P$.

Let $U = P \otimes_\Z \Q$ be the corresponding rational quadratic space containing these lattices.
The general spin group $T(\adeles_f) = \GSpin_U(\adeles_f)$, 
a central extension of the special orthogonal group, acts transitively on the lattices in the genus of $P$.

We describe this action in detail in \cref{sec:GspinU}.
We let $K \subset T(\adeles_f)$ be an open compact sugroup that preserves $P$ and acts
trivially on $P'/P$, where $P'$ is the dual lattice of $P$.
Consider the class group
\[
\Cl(K) = H(\Q) \bs  H(\adeles_f) / K
\]
and define a theta function on the product of the complex upper half-plane $\uhp$
and $\Cl(K)$ as
\[
   \Theta_P(\tau,h) = \sum_{\beta \in P'/P} \sum_{\lambda \in h(P+\beta)} e(Q(\lambda)\tau) \frake_\beta,
\]
where $e(x) = e^{2 \pi i x}$ and $\frake_\beta$ denotes the standard basis element of the group ring $\C[P'/P]$
corresponding to $\beta \in L'/L$. We will frequently write $\frake_0$ for $\frake_{0+P}$. 
We define the space $\Theta(P)$ of theta functions associated with $P$ to be the complex vector space generated by the forms $\Theta_{P}(\tau,h)$ for $h \in \Cl(K)$.
It is a subspace of $M_{1,P}$, the space of modular forms of weight $1$ and representation $\rho_P$.
Recall the definition of the Petersson inner product $(f,g)$, where $f,g$ are both modular forms of the same weight $k$ and representation $\rho_P$ as
\[
  (f,g) = \int_{\SL_2(\Z) \bs \uhp} \langle f(\tau), \overline{g(\tau)} \rangle v^k \frac{dudv}{v^2},
\]
where $\tau = u + iv$ with $u,v \in \R$ and $\langle \cdot, \cdot \rangle$ denotes the bilinear pairing
on $\C[P'/P]$, such that $\langle \frake_\mu, \frake_\nu \rangle = \delta_{\mu,\nu}$.
The integral converges if at least one of $f$ and $g$ is a cusp form.

It is useful to consider the following linear combinations of theta functions in $\Theta(P)$.
Let $\psi$ be a character of $\Cl(K)$. We let
\[
    \Theta_{P}(\tau,\psi) = \sum_{h \in \Cl(K)} \psi(h) \Theta_{P}(\tau,h).
\]
Our focus lies on lattices of the following form.
Let $P = \fraka$ be a fractional ideal in the imaginary quadratic field $k_D$ of discriminant $D<0$
with quadratic form $\N(x)/\N(\fraka)$. Moreover, let $K = \ODhatt = (\calO_D \otimes_\Z \hat\Z)^\times$,
where $\calO_D$ is the ring of integers of $k_D$ and $\hat\Z = \prod_p \Z_p$.
In this case, the group $\Cl(K)$ is isomorphic to the class group $\Clk$ of $k$.
\begin{theorem}\ \label{thm:thetaPbasis}
For $P$ and $K$ as above we have:
  \begin{enumerate}
  \item If $\psi = 1$, then  $\Theta_{P}(\tau,\psi) = E_{P}(\tau)$ is an Eisenstein series,
    spanning the space of Eisenstein series of weight one and representation $\rho_P$.
  \item If $\psi \neq 1$, then $\Theta_{P}(\tau,\psi)$ is a cusp form.
  \item 
Choose a system $\calC$ of representatives of characters on $\Clk$ modulo complex conjugation.
 Then the set
  \[
    \calB(P) = \{\Theta_{P}(\tau,\psi)\, \mid\, \psi \in \calC \}
  \]
  is an orthogonal basis for $\Theta(P)$.
 \item In particular, the dimension of the space $\Theta(P)$ is equal to
   \[
      \dim \Theta(P) = \frac{h_k+2^{t-1}}{2},
   \]
   where $h_{k}$ is the class number of $k$ and $t$ is the number of prime
   divisors of $D$.
  \end{enumerate}
\end{theorem}

\begin{remark}
  Note that the set $\calB(P)$ does depend on the choice of representatives, but only up
  to scalar factors.
\end{remark}

We also give an explicit formula for the Petersson inner products of
these basis elements in terms of special values of Dedekinds $\eta$-function. 
Recall that given an ideal $\frakb$ of $k$ which corresponds to the binary quadratic form $[a,b,c]$,
there is a CM point given by the unique root of the polynomial $a\tau^{2}+b\tau+c$
that lies in $\uhp$. We write $\tau(\frakb) = u(\fraka) + i v(\fraka) \in \uhp$ for this point.
\begin{proposition}
  \label{prop:thetapet}
 Let $\chi$ and $\psi$ be characters of $\Cl(K)$, not both trivial.
  With the assumptions of \cref{thm:thetaPbasis}, the following holds.
  \begin{enumerate}
  \item 
We have
  \[
    (\Theta_P(\tau,\psi),\Theta_P(\tau,\chi)) = 0
  \]
  unless $\psi = \bar{\chi}$ or $\psi = \chi$.
\item If $\psi^{2} \neq 1$ and $\psi = \bar{\chi}$, we obtain
  \[
    (\Theta_P(\tau,\psi),\Theta_P(\tau,\bar{\psi}))
    = - \psi(\fraka) h_k \sum_{\frakb \in \Clk} \psi(\frakb) \log \abs{v(\frakb)\eta^4(\tau(\frakb))},
  \]
\item   and if $\psi^{2} \neq 1$ but $\psi = \chi$,
  we have
  \[
    (\Theta_P(\tau,\psi),\Theta_P(\tau,\psi))
    = - h_k \sum_{\frakb \in \Clk} \psi(\frakb) \log \abs{v(\frakb)\eta^4(\tau(\frakb))}.
  \]
\item If $\psi = \chi$ and $\psi^{2} = \chi^{2} = 1$, the result is the sum of these two expressions.
  \end{enumerate}
\end{proposition}
The analogous formula is well known in the scalar valued case (see \cref{cor:scalar-theta-pet-expl}).
However, our proof is very different from the classical one that uses Kronecker's limit formula (see Proposition 3.1 in \cite{dukeliweightone}  for a proof).
We prove the formulas by using a certain seesaw identity and expressing the Petersson inner products
as CM values of a regularized theta lift.
This principle in fact generalizes to arbitrary dimensions, which will be the subject of a sequel to this article.

\section*{Acknowlegdements}
I would like to thank Jan Bruinier for his constant support and helpful comments on an earlier version of this paper.

\section{Shimura varieties for quadratic spaces of type $(2,0)$}
\label{cha:binary}
We let $k = k_D = \Q(\sqrt{D})$ be the imaginary quadratic field
of discriminant $D$ and we write $\Clk$ for the ideal class group of $k$.
We let $\OD \subset k$ be the ring of integers in $k$.
We write $\adeles_{k}$ for the ring of adeles over $k$.
Recall that \emph{idele class group} of $k$ is defined as the quotient
  \[
    \calI_k =  k^{\times} \bs \adeles_{k}^{\times}.
  \]
  Here, $k^{\times}$ is embedded diagonally into $\adeles_{k}^{\times}$ and the elements of
  the subgroup $k^{\times}$ are called \emph{principal ideles}.
  We also write
  \[
    I_{k} = k^{\times} \bs \adeles_{k,f}^{\times}
  \]
  for the finite idele class group.

\begin{theorem}[VI. Satz 1.3, \cite{neukirchalgzt}]
  \label{thm:idck}
  We have a surjective homomorphism
  \[
    I_{k} \rightarrow \Clk, \quad (\alpha_{\frakp})_{\frakp} \mapsto \prod_{\frakp \nmid \infty} \frakp^{v_{\frakp}(\alpha)},
  \]
  inducing an isomorphism
  \[
    I_{k} / \ODhatt = k^{\times} \bs \adeles_{k}^{\times} / \ODhatt \cong \Clk,
  \]
  where $\ODhatt$ is the subgroup
  \[
    \ODhatt = \prod_{\frakp \nmid \infty} \calO_{\frakp}^{\times}.
  \]
\end{theorem}

\subsection{Binary quadratic forms and ideals}
Recall that a fractional ideal $\fraka$ of $k$ defines an integral binary quadratic form
in the following way. If $\fraka$ is generated as a $\Z$-module by two elements
\[
 \fraka = \left( \alpha, \beta \right) = \Z \alpha + \Z \beta,
\]
then
\[
  Q_{\fraka}(x,y) = \frac{\N(\alpha)}{\N(\fraka)}x^2
     + \frac{\tr(\alpha\bar{\beta})}{\N(\fraka)}xy + \frac{\N(\beta)}{\N(\fraka)}y^2
     = \frac{\N(x\alpha+y\beta)}{\N(\fraka)}
\]
is an integral binary quadratic form of discriminant $D$.
This induces a bijective correspondence between equivalence classes of
positive definite integral binary quadratic forms of discriminant $D$ and the class group $\Clk$
of $k$ (if we also restrict to oriented bases).
A good reference for this correspondence is \cite{zagier-qf}.

\subsection{The action of $\GSpin_U(\adeles_f)$}
\label{sec:GspinU}
In this section, we let $U = \fraka \otimes_{\Z} \Q$
for a fractional ideal $\fraka$ of $k$ and we view $U$ simply as a 2-dimensional
rational quadratic space with quadratic form $Q(x) = \N(x)/\N(\fraka)$.
We write $T = \GSpin_{U}$. Over $\Q$ we have that $C_U^0 \cong \Q(\sqrt{-\abs{\det(U)}})$,
the even part of the Clifford algebra, is isomorphic to $k$.
The Clifford norm corresponds to the norm $\N(x)=x\bar{x}$ in $k$.
Here, $\bar{x}$ denotes complex conjugation.
Moreover, the group $SO_U(\Q)$ is isomorphic to
\[
  k^1 = \{ x \in k\, \mid\, \N(x)=1\}
\]
and $T(\Q) \cong k^{\times}$ is the multiplicative group of $k$.

Under this identification the map $T(\Q) \mapsto \SO_{U}(\Q)$ is given by $x \mapsto x/\bar{x}$.
This is essentially Hilbert's theorem 90 but can also be seen directly
by a short calculation using the definition of the Clifford group.
To see this, we consider the orthogonal basis $\{v_1=\N(\fraka), v_2 = -\sqrt{D}\}$ of $k$
as a vector space over $\Q$, where $D$ is the discriminant of $k$.
We have $Q(v_{1}) = \N(\fraka)$ and $Q(v_{2}) = \N(\sqrt{D})/\N(\fraka) = -D/\N(\fraka)$.

The even Clifford algebra $C_U^0$ is generated (as a $\Q$-algebra)
by $1$ and $\delta = v_{1}v_{2}$. Note that $\delta^2 = D$.

The group $\GSpin_{U}$ is given by all non-zero elements in $C_{U}^{0}$ in our case.
By definition, an element $a+b\delta \in \GSpin_{U}$ acts on $x \in k = U$ via
\[
  (a+b\delta) \cdot x \cdot (a+b\delta)^{-1}
\]
where the multiplication is in the Clifford algebra $C_{U}$.
It is enough to compute this on the basis vectors $v_{1},v_{2}$ of $k$.
It is easy to see that
$\delta v_{j}^{-1} = - v_{j}^{-1} \delta$ and we obtain
\[
  (a+b\delta) \cdot v_{j} \cdot (a+b\delta)^{-1}
  = (a+b\delta) \cdot ((a+b\delta)v_j^{-1})^{-1}
  = (a+b\delta) \cdot (a-b\delta)^{-1} \cdot v_{j}.
\]
The element $x=(a+b\delta) \cdot (a-b\delta)^{-1}$ is contained in $k^{\times}$.
The isomorphism $k^{\times} \cong T(\Q)$ is explicitly given via
$a + b \sqrt{D} \mapsto a+b\delta$.
Under this identification, the action of $x \in k^{\times} \cong \GSpin_{U}$
on $k$ is given by multiplication with $x/\bar{x}$.

Using this, we see that
$T(\adeles_{f}) \cong \adeles_{k,f}^{\times}$ is isomorphic to the multiplicative
group of ideles over $k$. To avoid confusion, in this section we write $h.x$ for the action of
$h \in T(\adeles_{f})$ on $x$ and simply $hx$ for multiplication of adeles.
Recall that the group $T(\adeles_{f}) = \GSpin_{U}(\adeles_{f})$ acts on lattices in $U$.
If $h = (h_p)_p \in \GSpin_{U}(\adeles_{f})$ and $L = \hat L \cap V(\Q)$ is a lattice in $V$, then $h.L = (h. \hat L) \cap V(\Q) = \prod_{p} (h_p. L_p)_p \cap U(\Q)$.

In the following, we will examine the action of $T(\adeles_{f})$ on
lattices in $U$ more closely. It is important to note that this action is
different form the ``natural'' action on ideals (or lattices) in $k$.
Recall that this natural action is simply given by the linear action of
$\Q_{p}^{\times}$ on $k \otimes \Q_{p}$.

The $\Q_{p}$ vector space $k \otimes_{\Q} \Q_p$ is an algebra with the multiplication
$(a \otimes b)(c \otimes d) = ac \otimes bd$, isomorphic to $C_{U}^{0}(\Q_{p})$.
It is also isomorphic \cite[II, Theorem 8.3]{neukirchalgzt} to the product
\begin{equation}
  \label{eq:12}
  \prod_{\frakp \mid p} k_{\frakp},
\end{equation}
where the product is over all prime ideals $\frakp$ of $k$ that lie above $p$.

For our purposes, it is enough to consider a lattice given by
a fractional ideal $\fraka \subset k$.
Then the action of $x \in T(\adeles_{f})$ is given as follows.

We write $\pi_{\frakp} \in \calO_{\frakp}$ for a uniformizer in $\calO_{\frakp}$.
This means that the only prime ideal in $\calO_{\frakp}$ is generated by $\pi_{\frakp}$
and every element in $k_{\frakp}$ can be written as $\pi_{\frakp}^{m}u$, where
$m \in \Z$ and $u \in \calO_{\frakp}^{\times}$. We can write
\[
\fraka = \prod_{\frakp} \frakp^{v_{\frakp}(\fraka)}
       = (\pi_{\frakp}^{v_{\frakp}(\fraka)})_{\frakp} \cap k,
\]
where we view $k$ as diagonally embedded into $\adeles_{k,f}$.

\begin{lemma}
  Let $h = (h_{\frakp})_{\frakp} \in T(\adeles_f)$. Then we have
\begin{equation*}
  \label{eq:ha}
  h.\fraka = \prod_{\frakp} \frakp^{v_{\frakp}(\fraka)+\mu_{\frakp}(h)},
\end{equation*}
where
\[
  \mu_{\frakp}(h) =
  \begin{cases}
    0, &\text{if } \frakp = \bar\frakp\\
    v_{\frakp}(h_{\frakp})-v_{\bar\frakp}(h_{\bar\frakp}), &\text{otherwise}.
  \end{cases}
\]
\end{lemma}
\begin{proof}
For primes $\frakp$ with $\frakp = \bar\frakp$,
that is for inert and ramified primes,
the action of $T(\Q_{p})$ does not change the valuation $v_{\frakp}$.
In those cases $h_{p} \in T(\Q_{p}) \cong k_{\frakp}^{\times}$ acts by multiplication with
$h_{p}/\bar{h}_{p}$, where $\bar{h}_{p}$ denotes the image of $h_{p}$
under the non-trivial Galois automorphism of the extension $k_{\frakp}/\Q_{p}$.

If the rational prime $p$ however splits in $k$
as $p\OD=\frakp \bar{\frakp}$, the action is
necessarily slightly different.
We have
\[
  k \otimes \Q_{p} \cong k_{\frakp} \times k_{\bar{\frakp}} \cong \Q_{p}^{2},
\]
as in \cref{eq:12}.
The isomorphism is given explicitly as follows.
Let $\frakd \in \Q_{p}$ with $\frakd^{2} = D$. Such a square-root exists
because $p$ is split in $k$ and therefore $D$ is a square modulo $p$.
Then the isomorphism  $k \otimes \Q_{p} \cong \Q_{p}^{2}$ is realized by
\[
  (a + b \sqrt{D}) \otimes c \mapsto ((a+b\frakd)c, (a-b\frakd)c) \in \Q_{p}^{2}.
\]
Therefore,
\[
  T(\Q_{p}) \cong (k \otimes \Q_{p})^{\times}
      \cong k_{\frakp}^{\times} \times k_{\bar\frakp}^{\times}
      \cong \Q_{p}^{\times} \times \Q_{p}^{\times}.
\]
Using the same arguments as over $\Q$, we see that
an element $x \otimes c \in T(\Q_p)$
acts by multiplication with $x/\bar{x} \otimes 1$.
Therefore, if $x = a + b\sqrt{D}$, then this corresponds to
multiplication with
\[
  \left(\frac{a+b\frakd}{a-b\frakd}, \frac{a-b\frakd}{a+b\frakd} \right),
\]
giving the formula in the lemma.
\end{proof}

In particular, we see that the action of $T(\adeles_{f})$ on lattices in $U$
is really fundamentally different from multiplication in the class group.
We denote the class of $h$ under the surjective map $\adeles_{k,f}^{\times} \rightarrow \Clk$
in \cref{thm:idck} by $[h]$. From the formulas above, we see that the action of $T(\adeles_{f})$
on the class $[\fraka]$ of $\fraka$ corresponds to multiplication by the class $[h]/\overline{[h]}$.
Here, $\overline{[h]}$ denotes the complex conjugate class of $[h]$.
Note that $[h]/\overline{[h]} = [h]^2$ since in an imaginary quadratic field
the ideal $\frakp\bar\frakp$ is a principal ideal for all prime ideals $\frakp \subset \OD$.
(It is either generated by $p=\N(\frakp)$ or by $p^{2}$.)

Therefore, the class $[h.\fraka] \in \Clk$ is given by $[h.\fraka] = [h]^2 [\fraka]$,
in accordance with the fact that $\GSpin_{U}(\adeles_{f})$ acts on lattices in the same genus.

\section{Petersson inner products of theta functions}
\label{sec:more-pp}
\subsection{Regularized theta lifts}
\label{sec:regul-theta-lifts}
We briefly recall Borcherds' regularized theta lift \cite{boautgra}.
We refer to the literature for details \cite{kudla-integrals, boautgra, brhabil, ehlen-diss}.

Let $V$ be a rational quadratic space with quadratic form $Q$ of signature $(b^+, b^-)$ and let $H=\GSpin_V$.
We write $(x,y) = Q(x+y)-Q(x)-Q(y)$ for the associated bilinear form.
Let $L \subset V$ be an even lattice and denote by $\Theta_L(\tau,z,h)$ the Siegel theta function associated with $L$.
For an appropriate choice of an open compact subgroup $K \subset H(\adeles_f)$, it is a function in $(z,h)$ on
the Shimura variety with complex points
\[
X_K(\C) = H(\Q) \bs ( \domain \times H(\adeles_f)/K),
\]
where $\domain$ is the symmetric space attached to $V$.
Assume that the signature $b^+-b^-$ is even. 
Recall that $\SL_2(\Z)$ is generated by
\[
S = \Smatrix \quad \text{and} \quad T = \Tmatrix.
\]
There is a unitary representation $\rho_L$ of $\SL_2(\Z)$
on the group ring $\C[L'/L]$, called the Weil representation.
The action of $\rho_L$ is defined as follows:
\begin{align*}
  \rho_L(T)\frake_\mu &= e(Q(\mu)) \frake_\mu,\\
  \rho_A(S)\frake_\mu &= \frac{ e((b^+ - b^-)/8)}{\sqrt{|L'/L|}}
  \sum_{\nu \in L'/L} e(-(\mu,\nu)) \frake_\nu.
\end{align*}
We write $M_{k,L}$ for the complex vector space of modular forms
of weight $k$ and representation $\rho_L$. Moreover, cusp forms are denoted $S_{k,L}$ and weakly holomorphic modular forms (which are allowed to have a pole at the cusp at $\infty$) by $M^!_{k,L}$.
Put $k=b^+-b^-$.
As a function of $\tau$, the theta function $\Theta_L(\tau,z,h)$ is a vector valued (non-holomorphic unless $L$ is positive definite) modular form of weight $k$ and representation $\rho_L$.

Denote by $\calF := \{\tau \in \uhp; \abs{\tau} \geq 1,\, -1/2 \leq \Re(\tau) \leq 1/2 \}$ the standard fundamental domain for the action of $\SL_2(\Z)$ and let $\calF_{T}:=\{\tau \in \calF;\ \Im(\tau) \leq T\}$.
Here and throughout, we write $d\mu(\tau) = dudv/v^2$ for $\tau = u+iv \in uhp$.
For a vector valued modular form $f \in M^!_{k,L}$, let
\[
\Phi_L(z,h,f) = 
\int_{\Gamma \backslash \uhp}^{\reg}
   \langle f(\tau), \overline{\Theta_L(\tau,z,h)} \rangle v^k d\mu(\tau)
  := \CT_{s=0} \left[
  \lim_{T\rightarrow \infty}
   \int_{\calF_T} \langle f(\tau), \overline{\Theta_L(\tau,z,h)} \rangle v^{k-s} d\mu(\tau) \right].
\]
Here, $\CT\limits_{s=0}$ denotes the constant term in the Laurent expansion at $s=0$
of the meromorphic continuation of the function in brackets defined by the limit.

\subsection{Special values of a theta lift and inner products}
We will now obtain an explicit expression for the Petersson inner products of the cusp forms contained in $\Theta(P)$.
We will utilize a seesaw identity that relates these inner products to special values of the Borcherds lift for $\Og(2,2)$.

Suppose that we are given a lattice $P$ of signature $(2,0)$ that corresponds
to the integral binary quadratic form $[A,B,C]$ of negative
fundamental discriminant $D \equiv 1 \bmod{4}$.
Equivalently, $P$ corresponds to an integral ideal $\fraka \subset \OD$
generated by $A$ and $(B+\sqrt{D})/2$.
Here, $\OD \subset k$ is the ring of integers in $k = \Q(\sqrt{D})$.

The lattice $P \oplus P^{-}$ has type $(2,2)$ and level $\abs{D}$.
We write $P^{-}$ for the lattice given by $P$ together with the negative of the quadratic form.
The discriminant group has order $D^{2}$. We take a $\Z$-basis $\{p_1,p_2\}$ of $P$
with $Q(p_1)=A$, $Q(p_2)=C$ and bilinear form $(p_1,p_2) = B$. We use the same basis for
$P^{-}$. The starting point is the following embedding.

Consider the even unimodular lattice $L=M_2(\Z)$ with the quadratic form given by $Q(X)=-\det(X)$.
The bilinear form is
\[
  (X,Y) = -\tr(XY^{\ast}), \text{ where } \twomat{a}{b}{c}{d}^{\ast} = \twomat{d}{-b}{-c}{a}
\]
and the type of $L$ is $(2,2)$.

The symmetric domain $\domain$ attached to $H = \GSpin_{V}$
can be identified with $\uhp^{2} \cup \bar\uhp^{2}$ in this case via
\begin{equation}
  \label{eq:h2isom}
  (z_1,z_2) \mapsto \R \Re \twomat{z_{1}}{-z_{1}z_{2}}{1}{-z_{2}} \oplus \R \Im \twomat{z_{1}}{-z_{1}z_{2}}{1}{-z_{2}}.
\end{equation}
\begin{lemma}
  Under this identification, the group $H=\GSpin_{V}$ for $V = L \otimes \Q$ can be identified
  with the subgroup $G$ of $\GL_2 \times \GL_2$ defined by
  \[
    G = \{ (g_1,g_2) \in \GL_2 \times \GL_2\ \mid\ \det g_1 = \det g_2 \},
  \]
  which acts on $\domain$ via fractional linear transformations in both components.
  The corresponding action of $(g_1,g_2) \in H$ on $x \in M_2(\Q)$ is given by
  \[
    (g_1,g_2).x = g_1xg_2^{-1}.
  \]
\end{lemma}
\begin{proof}
  Consider the orthogonal basis
  \[
    v_0 = \twomat{1}{0}{0}{1} = I_2,\ v_1 = \twomat{1}{0}{0}{-1},\ 
    v_2 = \twomat{0}{1}{1}{0},\ v_3 = \twomat{0}{1}{-1}{0}.
  \]
  According to Example 2.10 in the second contribution to \cite{123}, we have that
  the center $Z(C_{V}^{0})$ of the even Clifford algebra is given by $\Q \oplus \Q$
  and
  \[
    C_V^{0} = Z + Zv_1v_2 + Zv_2v_3 + Zv_1v_3.
  \]
  We obtain an isomorphism
  \[
    C_V^0 \cong  M_2(\Q) \oplus M_2(\Q)
  \]
  via
  \[
    1 \mapsto (I_2,I_2), \quad v_{i}v_{j} \mapsto (v_iv_j^{*}, v_iv_j^{*}).
  \]
  Under this isomorphism, the canonical involution of $C_{V}$ corresponds
  to
  \[
    (A,B) \mapsto (A^{*},B^{*})
  \]
  and the Clifford norm is given by
  \[
    \N(A,B) = (\det(A)I_2,\det(B)I_2).
  \]
  Therefore, $\N(A,B) \in \Q^{\times}$
  is equivalent to $A,B \in \GL_2(\Q)$ with $\det(A) = \det(B)$.

  It is straifghtforward to check that under the identification \cref{eq:h2isom},
  the action of $H$ corresponds to fractional linear transformations.
\end{proof}
We let $K = H(\Zhat)$, that is
\[
  K = H(\Zhat)
  = \{ (g_1,g_2) \in \GL_2(\Zhat) \times \GL_2(\Zhat)\ \mid\ \det g_1 = \det g_2 \in \Zhat \}.
\]
It is clear that $K$ preserves $L$.
By strong approximation and the theory of Shimura varieties,
the associated Shimura variety $X_K$ is a product of two modular curves
\[
  X_{K} = H(\Q) \bs \domain \times H(\adeles_f) / K \cong \SL_{2}(\Z) \bs \uhp \times \SL_{2}(\Z) \bs \uhp.
\]
It turns out that the additive Borcherds lift of the constant function is equal to
\begin{equation}
  \label{eq:Boeta}
  \Phi_{L}(z_{1},z_{2},1) = - 4 \log\abs{(y_{1}y_{2})^{1/4}\eta(z_{1})\eta(z_{2})} -\log(2\pi) - \Gamma'(1)
\end{equation}
as a function on $\uhp^{2}$.
We refer to Section 5.1 of the thesis of Hofmann \cite{ericdiss} for details.

Consider the point
\[
  z_0 = \left( \frac{-B+\sqrt{D}}{2A}, \frac{-B+\sqrt{D}}{2} \right) \in \uhp^2.
\]
It corresponds to the two rational points $z_{P}^{\pm} \in \domain$, as we shall see below.
For simplicity, we drop the sign $\pm$ indicating the orientation from our notation.

\renewcommand{\arraystretch}{1.3}
A basis of $z_{0} \cap V(\Q)$ is given by
\[
  \Q f_1 \oplus \Q f_2, \text{ with } f_1 = \twomat{-1}{-B}{0}{A} \text{ and } f_2 = \twomat{0}{-\frac{B^{2}-D}{4A}}{-1}{0}.
\]
Indeed, we have
\[
  z_{0} = \R
  \begin{pmatrix}
    \frac{-B}{2A} & -\frac{B^2+D}{4A} \\
    1            & \frac{B}{2}
  \end{pmatrix}
          \oplus \R
  \sqrt{\abs{D}}
  \begin{pmatrix}
           \frac{1}{2A} & \frac{B}{2A} \\
           0                         & -\frac{1}{2}
  \end{pmatrix}.
\]
We obtain $f_1$ as
\[
  f_1 = -2A \begin{pmatrix}
          \frac{1}{2A} &    \frac{B}{2A} \\
          0                         & -\frac{1}{2}
        \end{pmatrix}.
\]
and
\[
  f_2 = -\begin{pmatrix}
          \frac{-B}{2A} & -\frac{B^2+D}{4A} \\
          1            & \frac{B}{2}
        \end{pmatrix}
        - B
       \begin{pmatrix}
          \frac{1}{2A} & \frac{B}{2A} \\
          0                         & -\frac{1}{2}
        \end{pmatrix}.
\]
In fact, with this choice of basis, we get an isometry of even lattices.
\begin{lemma}
  We have an isometry of lattices $(P,Q) \cong \Z f_1 \oplus \Z f_2 \subset L$
  given by
  \[
    p_1 \mapsto f_1, \quad p_2 \mapsto f_2,
  \]
  or, equivalently of $(\fraka, \N(x)/\N(\fraka)) \cong (P,Q)$ given by
  \[
    A \mapsto f_1, \quad \frac{B+\sqrt{D}}{2} \mapsto f_2.
  \]
  Moreover, we have for $U = \Q f_1 \oplus \Q f_2$ that
  $L \cap U = \Z f_1 \oplus \Z f_2 = P$ and 
  \[
    L \cap U^{\perp} = \Z
      \begin{pmatrix}
        1 & 0 \\ 0 & A
      \end{pmatrix}
      \oplus
      \Z
      \begin{pmatrix}
        0 & -\frac{B^{2}-D}{4A} \\ 1 & B 
      \end{pmatrix}
  \]
  is isometric to $P^{-}$.
\end{lemma}
\begin{proof}
  It is trivial to check that $Q(f_1) = A$, $Q(f_2) = (B^2-D)/4A$ and $(f_1,f_2)=B$.
  Similarly, the matrices
  \[
    \tilde{f}_{1} =
     \begin{pmatrix}
        1 & 0 \\ 0 & A
      \end{pmatrix}, \quad
    \tilde{f}_{2} =
    \begin{pmatrix}
        0 & -\frac{B^{2}-D}{4A} \\ 1 & B 
    \end{pmatrix}
  \]
  satisfy  $Q(\tilde{f}_1) = -A$, $Q(\tilde{f}_2) = -(B^2-D)/4A$ and $(\tilde{f}_1,\tilde{f}_2) = -B$.
  Moreover, $f_1$ and $f_2$ are both orthogonal to $\tilde{f}_1$ and $\tilde{f}_2$.

  As for the equalities $L \cap U = \Z f_1 \oplus \Z f_2$ and $L \cap U^{\perp} = \Z \tilde{f}_1 \oplus \Z \tilde{f}_2$,
  the inclusions ``$\subset$" are clear and the other direction is easy to see
  because any non-integral linear combination of these vectors has a non-integral entry.
\end{proof}
The lemma provides an embedding of $P \oplus P^{-}$ into $L$ as an orthogonal sum.
Under this embedding, $z_P = z_U = z_0$.
We write $T = \GSpin_{U}$ and identify it with $k^{\times}$ as an algebraic group
over $\Q$, as before. We now come to the corresponding embedding
on the level of orthogonal groups.
Note that we have $K_T := K \cap T(\adeles_f) \cong \ODhatt$.
Recall that given an ideal $\frakb$ of $k$ which corresponds to the binary quadratic form $[a,b,c]$,
there is a CM point given by the unique root of the polynomial $a\tau^{2}+b\tau+c$
that lies in $\uhp$.  We write $\tau(\frakb) = u(\fraka) + i v(\fraka) \in \uhp$ with $u(\fraka), v(\fraka) \in \R$ for this point.
\begin{lemma}
  \label{lem:TTembed}
  The group $T = \GSpin_{U}$ embeds into $G$ via
  \[
    1 \mapsto \left( I_2, I_2 \right),
  \]
  where $I_2 \in \GL_2$ is the identity matrix
  and
  \[
    \sqrt{D} \mapsto \left( X,Y \right), \text{ where } X =
    \begin{pmatrix}
      -B & \frac{D-B^{2}}{2A}\\ 2A & B
    \end{pmatrix}
    \text{ and }
     Y = \begin{pmatrix}
      -B & \frac{D-B^{2}}{2}\\ 2 & B
    \end{pmatrix}.
  \]
  Similarly, the image of $T^{-}=\GSpin_{U^{\perp}}$ is given by
  \[
    \sqrt{D} \mapsto \left( X,Y^{\ast} \right).
  \]
\end{lemma}
\renewcommand{\arraystretch}{1}
\begin{proof}
  This can easily be seen by determining the stabilizer
  of the point $z_{P}$ as given above on $\uhp^{2}$.
  We also refer to Section 4.4 in \cite{shimauto}.
  Proposition 4.6, ibid., tells us that if $\C/\Lambda$ is an
  elliptic curve with complex multiplication, $\Lambda = \Z + \Z\tau$,
  then there is an embedding $q_\tau$ of $k$ into $M_2(\Q)$,
  such that
  \[
    q_{\tau}(k^{\times}) = \{A \in \GL_2^{+}(\Q) \mid A\tau = \tau\}.
  \]  
  There are exactly two embeddings with this property for a given point $\tau$.
  One of them has the property
  \[
    q_{\tau}(\mu)\colvec{\tau}{1} = \mu \colvec{\tau}{1}.
  \]
  The other one, denoted $\bar{q}_{\tau}$, satisfies the same property with $\tau$ replaced by $\bar{\tau}$,
  that is,
  \[
   \bar{q}_{\tau}(\mu)\colvec{\tau}{1} = \bar{\mu} \colvec{\tau}{1}.
  \]

  It is easy to check that $q_{\tau(\fraka)}(\sqrt{D})=X$ and $q_{\tau(\OD)}(\sqrt{D})=Y$
  as well as $\bar{q}_{\tau(\OD)}(\sqrt{D})=Y^{*}$.
  Using these formulas, we see that the correct embedding of $k^\times \times k^\times$ in our case is given by
  $(\lambda,\mu) \mapsto (q_{\tau(\fraka)}(\lambda)q_{\tau(\fraka)}(\mu),q_{\tau(\OD)}(\lambda)\bar{q}_{\tau(\OD)}(\mu))$
  for $\lambda,\mu \in k$.
\end{proof}

Similar to the proof of Theorem 6.31 in \cite{shimauto}, we have a commutative diagram
\[
\xymatrix{
   \Q^{2} \ar[r]^{\iota_{\tau}} \ar[d]_{q_{z}(\mu)} & k^{\times}  \ar[d]^{\mu}\\
   \Q^{2} \ar[r]^{\iota_{\tau}}                     & k^{\times},
}
\]
where 
$$
\iota_{\tau}(x_1,x_2) = (x_{1},x_{2}) \colvec{\tau}{1}
$$ 
and the vertical arrow
on the right is given by multiplication with $\mu$.
The map $q_{\tau}$ extends to $\adeles_{k,f}^{\times}$ and $q_{\tau}(\adeles_{f}) \subset \GL_{2}(\adeles_{f})$
acts on lattices in $\Q^{2}$. Similarly, we have the linear action given by an idele on the right
and these actions commute with the map $\iota_{\tau}$ in the same way.
We let $\fraka_{\tau} = \Z \tau + \Z$ and $q_{\tau}(h) = \gamma k$ for $h \in \adeles_{k,f}^{\times}$
and with $\gamma \in H(\Q)$ and $k \in K$. There is an element $\mu \in k^{\times}$, such that
\[
  \gamma^{-1}\colvec{\tau}{1} = \mu \colvec{\tau}{1}.
\]
Therefore, we have
\[
  h^{-1}\fraka_\tau = \iota_{\tau}(\Z^{2}q(h)^{-1}) = \iota_{\tau}(\Z^{2}\gamma^{-1}) = \mu \fraka_{w},
\]
where $w = \gamma^{-1}\tau$.

This shows that for $g \in T(\adeles_f)$ and $h \in T^{-}(\adeles_{f})$,
we have
\[
  H(\Q)((\tau(\fraka),\tau(\OD)), (g,h))K = H(\Q)((\tau((gh)^{-1}\fraka),\tau((g^{-1}h))),(1,1))K.
\]
Here, we used the notation $(h)$ for the ideal corresponding to $h$ and $(h)\fraka$
means multiplication of fractional ideals (and \emph{not} the action of $\GSpin_{U}$
on lattices in $U$).
\begin{proposition}
  \label{prop:PhiPeta}
  Let $g,h \in T(\adeles_f) \cong \adeles_{k,f}^{\times}$.
  We write $\tau_{1} = \tau((hg)^{-1}\fraka) = u_1 + iv_1$ and
  $\tau_2 = \tau((gh^{-1})) = u_2 + iv_2$ and obtain
  \[
   \Phi_{P}(\Theta_{P}(\tau,g),h) = - 4 \log\abs{(v_1v_2)^{1/4}\eta(\tau_1)\eta(\tau_2)} -\log(2\pi) - \Gamma'(1).
  \]
Note that the value depends only on the ideal classes of $(h)$, $(g)$ and $\fraka$.
\end{proposition}
\begin{proof}
  The proposition essentially follows from the identity
\begin{align*}
  - 4 \log\abs{(v_1v_2)^{1/4}\eta(\tau_1)\eta(\tau_2)} -\log(2\pi) - \Gamma'(1) &= \Phi_{L}((z_P,(h,g)),1),
\end{align*}
which is a consequence of \cref{eq:Boeta} and our considerations above as follows:
We use the maps $\res_{L/(P \oplus P^{-})}$ and $\tr_{L/(P \oplus P^{-})}$ defined in Lemma 3.1 in \cite{bryfaltings}.
Note that the Siegel theta function satisfies
\[
  \Theta_{P \oplus P^{-}}(\tau,(h,g)) = \Theta_{P}(\tau,h) \otimes \Theta_{P^{-}}(\tau,g)
\]
and $\Theta_{P \oplus P^{-}}^{L} = \Theta_{L}$.
Moreover, we have that
\begin{align*}
  \langle f(\tau), \overline{\Theta_{L}(\tau,z_{P},(h,g))} \rangle
  &= \langle f_{P \oplus P^{-}}(\tau), \overline{\Theta_{P}(\tau,h) \otimes \Theta_{P^{-}}(\tau,g)} \rangle\\
  &= \langle f_{P \oplus P^{-}}(\tau), \Theta_{P^{-}}(\tau,h) \otimes \Theta_{P}(\tau,g) \rangle.
\end{align*}

With the embeddings defined above, we consider $P \oplus P^{-}$ as a sublattice of $L$.
Then we have $P \oplus P^{-} \subset L = L' \subset P' \oplus (P^{-})^{'}$
and
\[
  L/(P \oplus P^{-}) \subset P'/P \oplus (P^{-})^{'}/P^{-} \cong P'/P \oplus P'/P.
\]
Using our embeddings defined above, it is not hard to see that for the constant function $1$, we have
\[
  1_{P \oplus P^{-}} = \res_{L/(P \oplus P^{-})}(1) = \sum_{\beta \in P'/P} \frake_{\beta + P} \otimes \frake_{\beta + P^{-}}.
\]
Thus, we obtain
\begin{align*}
      \Phi_{L}((z_P,(h,g)),1)
      &= \int_{\SL_2(\Z) \bs \uhp}^{\reg}
          \langle 1_{P \oplus P^{-}} , \Theta_{P^{-}}(\tau,h) \otimes \Theta_{P}(\tau,g) \rangle d\mu(\tau) \\
      &= \int_{\SL_2(\Z) \bs \uhp}^{\reg}
          \langle \Theta_{P}(\tau,g) ,\overline{\Theta_{P}(\tau,h)} \rangle v d\mu(\tau) \\
      &= \Phi_{P}(\Theta_{P}(\tau,g),h).\qedhere
\end{align*}
\end{proof}

\section{Proofs of Proposition 1.4 and Theorem 1.2}
\label{sec:vector-valued-theta}
As in the introduction, let $(P,Q)$ be a two-dimensional positive definite even lattice.
We let $U = P \otimes_{\Z} \Q$ be the associated rational quadratic space.
We will assume that $(P,Q)$ is given by a fractional ideal $\fraka$ in an imaginary number field $k$ as in the last sections
and only use the letter $P$ to distinguish between the scalar valued and vector valued case. We have that the dual lattice of $P$
is given by $P' \cong \different{k}^{-1}\fraka$, where $\different{k}$ denotes the different ideal of $k$.
Recall the definition of the theta function $\Theta_P(\tau,h)$ attached to $P$ from the introduction.

\begin{remark}
We should warn the reader that if $P=\fraka \subset k$ is a fractional
ideal, the theta function $\Theta_P(\tau,h)$ is in general not verbatim equal to
the vector-valued theta function corresponding to $(h)^{2}\fraka$,
if $(h)$ denotes the ideal corresponding to $h$ (defined as in \cref{thm:idck}).
This is due to the fact that $T(\adeles_{f})$ also acts on the components via automorphisms.
\end{remark}

We can prove the explicit expression in \cref{prop:thetapet} for the Petersson inner products of vector
valued theta functions in $\Theta(P)$ using \cref{prop:PhiPeta}.
\begin{proof}[Proof of \cref{prop:thetapet}]
  Let us abbreviate
  \[
    f(\frakb) = v(\frakb)^{1/4}\eta(\tau(\frakb))
  \]
  for any fractional ideal (class) $\frakb \subset k$.
  We have by definition and \cref{prop:PhiPeta} that
  \begin{align*}
    (\Theta_P(\tau,\psi),\Theta_P(\tau,\chi)) &= \sum_{h,g \in T(\adeles_f)/K_T} \psi(g) \bar\chi(h) \Phi_{P}(\Theta_{P}(\tau,g),h) \\
      &= - 4 \sum_{h,g \in T(\adeles_f)/K_T} \psi(g) \bar\chi(h) \log\abs{f((hg)^{-1}\fraka)f(\tau((h^{-1}g)))}
  \end{align*}
  because for non-trivial characters the constant does not contribute to the sum
  by orthogonality of characters.
  We split the sum above into 
  \begin{align*}
    &\sum_{g,h} \psi(g) \bar\chi(h) \log\abs{f((hg)^{-1}\fraka)} + \sum_{g,h} \psi(g) \bar\chi(h) \log\abs{f((h^{-1}g))} \\
    &\quad = \sum_{g} \psi(g) \chi(g) \sum_{h}\chi(h)\log\abs{f(h\fraka)}
              + \sum_{g} \psi(g) \bar\chi(g) \sum_{h}\chi(h)\log\abs{f(h)} \\
    &\quad =
    \begin{cases}
      h_{k} \sum_{h}\chi(h)\log\abs{f((h)\fraka)}, &\text{if } \psi = \bar\chi,\\
      h_{k} \sum_{h}\chi(h)\log\abs{f(h)}, &\text{if } \psi=\chi, \\
      0, & \text{otherwise},
    \end{cases}
  \end{align*}
  as long as we do not have $\chi = \psi = \bar{\psi}$, in which case we get the sum of the two terms.
  For the first sum, we obtain
  \begin{align*}
    \sum_{h}\chi(h)\log\abs{f((h)\fraka)}   &= \sum_{h} \chi(h) \log\abs{v((h)\fraka)^{1/4}\eta(\tau((h)\fraka))}\\
                                            &= \bar{\chi}(\fraka) \sum_{h} \chi(h) \log\abs{v((h))^{1/4}\eta(\tau((h)))}\qedhere
  \end{align*}
\end{proof}

We can now give the proof of \cref{thm:thetaPbasis}.

\begin{proof}[Proof of \cref{thm:thetaPbasis}]
  That $E_{P}(\tau)$ as defined above is really an Eisenstein series follows from the Siegel-Weil formula (Theorem 2.1 of \cite{bryfaltings}).
  The Eisenstein series correspond to isotropic vectors in the discriminant group $P'/P$ (see \cite{brhabil}) and we assumed that $\abs{P'/P}=\abs{D}$ is square-free, which implies (i).

  That $\Theta_{P}(\tau,\psi)$ is a cusp form for non-trivial
  $\psi$ is clear.

  To see that $\calB(P)$ is a basis of $\Theta(P)$, first note that
  \cref{prop:thetapet} implies that the set $\calB(P)$
  is linear independent.
  Moreover, if $\psi^2 \neq 1$ the Proposition also implies
  \[
    (\bar{\psi}(\fraka) \Theta_{P}(\tau,\psi) - \Theta_{P}(\tau,\bar{\psi}),f(\tau)) = 0
  \]
  for all $f \in \Theta(P)$.
  Therefore, $\bar\psi(\fraka) \Theta_{P}(\tau,\psi) - \Theta_{P}(\tau,\bar{\psi}) \in \Theta(P) \cap \Theta(P)^{\perp}$,
  where $\Theta(P)^{\perp}$ is the orthogonal complement of $\Theta(P)$
  with respect to the Petersson inner product. Consequently,
  $\Theta_{P}(\tau,\bar{\psi}) = \bar\psi(\fraka) \Theta_{P}(\tau,\psi)$.

  Finally, let $A$ be the set of elements $x \in \Clk$,
  such that $\bar{x} = x$ and let $B = \Clk \setminus A$.
  Then $\abs{\calB(P)} = \abs{A} + \abs{B}/2$.
  Moreover, it is well known that $\abs{A} = 2^{t-1}$ and $\abs{B} = h_{k} - 2^{t-1}$ which implies the assertion.
\end{proof}

\section{Liftings of newforms in the case of square-free level}
\label{sec:sym}
In this section we will show some general properties of
liftings of scalar valued modular forms to vector valued modular forms
in the case of square-free level. We will apply these results to
relate scalar valued theta series to vector valued ones.
This lifting has been used by Bundschuh in his thesis \cite{BundschuhDiss},
by Bruinier and Bundschuh \cite{BruinierBundschuh} and Scheithauer \cite{Scheithauer-liftings}.

Let $L$ be an even lattice with quadratic form $Q$ of type $(2,n)$, level $N$ and determinant $D = \abs{L'/L}$.
The group $\Gamma_0(N)$ acts on $\frake_{0}$ via the Weil
representation $\rho_{L}$ by a character. It is given by
\[
\chi_{L} \left(
  \begin{pmatrix}
    a & b \\ c & d
  \end{pmatrix}
\right) =
\begin{cases}
  \left( \frac{(-1)^{\frac{n+2}{2}}D}{d} \right)& \text{if } d>0,\\
  (-1)^{\frac{n+2}{2}} \left( \frac{(-1)^{\frac{n+2}{2}}D}{-d} \right)& \text{if } d<0.
\end{cases}
\]
We will throughout assume that $N$ is square-free. 
Then $2+n$ is even and the character is quadratic.
Moreover, this implies that for any $f \in M_{k,L}$, the component function
$f_{0}$ is a modular form in $M_{k}(N,\chi_{L})$. Conversely,
we can ``lift'' any $f \in M_{k}(N,\chi_{L})$ to a vector-valued
modular form by defining
\begin{equation}
  \label{liftdef}
  \S_{L}(f) = \sum_{\gamma \in \Gamma_{0}(N) \bs \SL_{2}(\Z)} (f \mid_{k} \gamma) \rho_{L}(\gamma^{-1})\frake_{0} \in M_{k,L}.
\end{equation}

There is also a map that is adjoint to the lift with respect
to the Petersson inner product. It is simply given by the map $F \mapsto F_{0}$
for $F \in M_{k,L}$. We also write $(f,g)$ for the Petersson inner product on the space of
cusp forms for $\Gamma_0(N)$ (possible with character), i.e.
\[
 (f,g) = \int_{\Gamma_0(N)\bs\uhp} f(\tau) \overline{g(\tau)}\Im(\tau)^{k} d\mu(\tau)
\]
for $f,g \in S_k(N,\chi)$.
\begin{proposition}
  \label{prop:pet-lift}
  Let $f \in S_{k}(N,\chi_L)$ and let $F \in M_{k,L}$. Then, we have
  for the Petersson inner product
  \[
  (\S_{L}(f),F) = (f,F_0).
  \]
\end{proposition}
\begin{proof}
  Using the definitions, we obtain
  \begin{align*}
    \label{eq:pet1}
    (\S_{L}(f),F) &= \int_{\calF} \langle \sum_{\gamma \in \Gamma_0(N) \bs \SL_{2}(\Z)}
    (f \mid_k \gamma) \rho_{L}(\gamma^{-1})\frake_0,
    \overline{F(\tau)} \rangle v^k d\mu(\tau)\\
    &= \int_{\calF} \sum_{\gamma \in \Gamma_0(N) \bs \SL_{2}(\Z)} (f \mid_k \gamma) \langle \frake_0,
    \overline{\rho_{L}(\gamma) F(\tau)} \rangle v^k d\mu(\tau)\\
    &= \int_{\calF} \sum_{\gamma \in \Gamma_0(N) \bs \SL_{2}(\Z)} \Im(\gamma\tau)^{k} f(\gamma \tau) \langle \frake_0,
    \overline{F(\gamma\tau)} \rangle d\mu(\tau)\\
    &= \sum_{\gamma \in \Gamma_0(N) \bs \SL_{2}(\Z)}
    \int_{\gamma\calF} \Im(\tau)^{k} f(\tau) \overline{F_0(\tau)}d\mu(\tau).
  \end{align*}
  The last line is equal to the Petersson inner product of $f$ and $F_0$ defined as in the statement of the Proposition.
\end{proof}

Following Bundschuh \cite{BundschuhDiss}, we define a subspace of the newforms in $S_{k}(N,\chi_{L})$.
Let $A = L'/L$ and for a prime $p$ denote by $A_{p}$ the $p$-component of $A$.
Moreover, write $\chi_L = \prod_{p \mid N} \chi_{L,p}$
as a product of characters modulo $p$ for $p \mid N$.
For each prime $p_{i}$ dividing $N = p_{1} \ldots p_{r}$, we define a an element
$\varepsilon_{i} \in \{0, 1, -1\}$.
\begin{definition}
  \label{def:epsilons}
  If $\dim_{\F_{p_{i}}} A_{p_{i}} \geq 2$ or $p_{i}=2$, we define $\varepsilon_{i}=0$.
  If $\dim_{\F_{p_{i}}} A_{p_{i}} = 1$, $p_{i} \neq 2$ and $N Q \mid_{A_{p_{i}}}$ represents the squares modulo $p_{i}$,
  we define $\varepsilon_{i} = 1$.
  Otherwise, we define $\varepsilon_{i}=-1$.
  Using these signs, we let
  \[
    S_{k}^{\varepsilon_{1},\ldots,\varepsilon_{r}}(N,\chi_{L})
      = \{ f \in S_{k}^{\new}(N,\chi_{L})\, \mid\, \exists\, i \text{ with } \varepsilon_i \neq 0 \text{ and } \chi_{L,p_{i}}(n) = -\varepsilon_{i} \Rightarrow c_{f}(n) = 0 \}.
  \]
\end{definition}

\begin{remark}
  Note that we have
  \[
    S_{k}^{\new}(N,\chi_{L}) =
    \bigoplus_{(\varepsilon_{1},\ldots,\varepsilon_{r}) \in \{\pm1\}^{r}} S_{k}^{\varepsilon_{1},\ldots,\varepsilon_{r}}.
  \]
  We refer to the thesis of Bundschuh \cite[Satz 4.3.4]{BundschuhDiss} for details.
\end{remark}

\begin{theorem}
  \label{thm:Lift0}
  Let $L$ be an even lattice of square-free level $N$ and
  $f \in S_{k}^{\varepsilon_{1},\ldots,\varepsilon_{r}}(N,\chi_{L})$.
  Assume that $\dim_{\F_{p_{i}}} A_{p_{i}} = 1$ or
  $\dim_{\F_{p_{i}}} A_{p_{i}} \geq 2$ even for all odd $p_{i}$.
We have
  \[
    \langle \calS_L(f), \frake_{0 + L} \rangle = \nu \frac{N}{\abs{L'/L}} f.
  \]
  Here, we let
 \[
  \nu = \nu(m) = \#\{\mu \in L'/L\, \mid\, NQ(\mu) \equiv m \bmod{N} \}
\]
  for any $m \in \Z$ with $(m,N)=1$ and $\nu(m) \neq 0$, which is independent of the choice of $m$.
\end{theorem}
\begin{proof}
  We follow the proof of Satz 4.3.9 in \cite{BundschuhDiss}.
  Let
  \[
    f(\tau) = \sum_{n =1}^{\infty} a(n) e(n\tau)
  \]
  be the Fourier expansion of $f$ (at the cusp $\infty$)
  and let
  \[
    f \mid_{k} W_{N} = \sum_{n =1}^{\infty} a_{N}(n) e(n\tau).
  \]
  Let $\mu \in L'/L$ with $(NQ(\mu),N)=1$.
  In this case it is not hard to see that
  \begin{equation}
    \label{eq:Fmulift}
    F_{\mu}(\tau) = \frac{N^{1-k/2}e(\sgn(L)/8)}{\sqrt{\abs{L'/L}}}
                 \sum_{n \equiv NQ(\mu) \bmod{N}} a_{N}(n) e\left( \frac{n}{N} \tau \right),
  \end{equation}
where $\calS_L(f) = F(\tau) = \sum_{\mu \in L'/L} F_\mu(\tau)\frake_\mu$.
  This follows from Theorem 4.2.8 in \cite{BundschuhDiss}
  and can also be deduced from explicit formulas for the Weil
  representation \cite{Scheithauer-formulas, fredrik-weilfqm}.
  We obtain
  \begin{align}
    F_{0} \mid_{k} W_{N}  &= N^{k/2} (F_{0} \mid_{k} S) (N\tau)
    = N^{k/2} \frac{e(-\sgn(L)/8)}{\sqrt{\abs{L'/L}}} \sum_{\mu \in L'/L} F_{\mu}(N\tau) \notag \\
                        &= \frac{N}{\abs{L'/L}} \sum_{0 \neq \mu \in L'/L}
                            \sum_{n \equiv NQ(\mu) \bmod{N}} a_{N}(n) e\left( n \tau \right)
                           + \sum_{\substack{n > 1 \\ (n,N)>1}} b(n) e(n\tau) \label{eq:F0WN},
  \end{align}
  with certain coefficients $b(n)$.

  By the assumptions of the theorem
  on the dimension of $A_{p}$ over $\F_{p}$ for $p \mid N$,
  we have that the representation number
  \[
    \nu(m) = \abs{\{\mu \in L'/L\, \mid\, NQ(\mu) \equiv m \bmod{N} \}}
  \]
  is in fact equal for all $m \neq 0$ with $\nu(m) \neq 0$
  (cf. \cite[Section 13]{kneser-qf}).
  Therefore, if we put $\nu = \nu(m)$ for any $m \in \Z$
  with $(m,N) = 1$ and $\nu(m) \neq 0$, the last expression simplifies to
  \begin{equation}
    \label{eq:13}
     F_{0} \mid_{k} W_{N} = \frac{N}{\abs{L'/L}} \nu \sum_{(n,N)=1} a_{N}(n) e\left( n \tau \right)
                           + \sum_{\substack{n > 1 \\ (n,N)>1}} b(n) e(n\tau),
  \end{equation}
  Here, we used the assumption that
  $f \in S_{k}^{\varepsilon_{1},\ldots,\varepsilon_{r}}(N,\chi_{L})$.
  Therefore, we can express the difference to $f \mid_k W_N$ as
  \[
    F_0 \mid_k W_N - \frac{N}{\abs{L'/L}} \nu f \mid_k W_N = \sum_{\substack{n \geq 1 \\ (n,N)>1}} c(n) e(n\tau)
  \]
  for some complex numbers $c(n)$.  
  However, we also have that $F_0$ is a newform (see for instance Proposition 7.3 in \cite{markus-fabian-liftings}).
  Thus, $F_{0} \mid_{k} W_N$ is a also newform, and hence the difference vanishes.
\end{proof}
The group $\Og(L'/L)$ acts on vector-valued modular forms
by permuting the basis vectors $\frake_{\mu}$. That is,
$\sigma \in \Og(L'/L)$ acts via $\frake_{\mu} \mapsto \frake_{\sigma(\mu)}$.
Using this action, we define the symmetrization of a modular form $f \in M_{k,L}$
as
\[
  f^{\sym} (\tau) = \sum_{\sigma \in \Og(L'/L)} f^{\sigma}(\tau)
  = \sum_{\mu \in L'/L} \sum_{\sigma \in \Og(L'/L)} f_{\mu} (\tau) \frake_{\sigma (\mu)}.
\]
This function is clearly invariant under the action of $\Og(L'/L)$.
We write $M_{k,L}^{\sym}$ for the subspace of $M_{k,L}$ that is invariant under
$\Og(L'/L)$. The map
\begin{equation}
  \label{eq:sym}
  M_{k,L} \longrightarrow M_{k,L}^{\sym}, \quad f \mapsto f^{\sym}
\end{equation}
is obviously surjective.

The following proposition can be found in Propositions 5.1 and 5.3 of \cite{Scheithauer-liftings}.
\begin{proposition}
  \label{prop:Scheit-sqf}
  Let $L$ be an even lattice of square-free level $N$.
  Then the orthogonal group $\Og(L'/L)$ acts transitively on all
  elements of the same norm and order in $L'/L$.
  Moreover, if $F \in M_{k,L}^{\sym}$ and $F_{0}=0$, then $F=0$.
\end{proposition}

\section{Lifting scalar valued theta functions}
\label{sec:binary-theta}
As before, let $D<0$ be an odd fundamental discriminant and let $k = \Q(\sqrt{D})$ be the imaginary quadratic field of discriminant $D$.
We write $\OD$ for the ring of integers in $k$ and $\Clk \cong \Cl(K)$ for the ideal class group of $k$
as in \cref{sec:vector-valued-theta}.
We assume that $D$ is odd.
\subsection{Scalar valued theta functions}
\label{sec:scalar-valued-theta}
For an integral ideal $\fraka \subset \OD$, we can consider the associated theta function
\begin{equation}
  \label{eq:stheta}
  \theta_{\fraka}(\tau) = \sum_{a \in \fraka} e\left(\frac{\N(a)}{\N(\fraka)} \tau \right)
                       = 1 + \sum_{n \geq 1} \rho(n,\fraka) e(n \tau).
\end{equation}

Since $Q_{\fraka}$ is positive definite, the series converges normally and defines a
holomorphic modular form of weight one. It is contained in
$M_{1}(\abs{D},\chi_{D})$, where $\chi_{D}$ is the
primitive Dirichlet character of conductor $\abs{D}$.

It is easy to see that the representation number $\rho(n,\fraka)$,
and therefore also the theta function $\theta_{\fraka}$,
only depends on the class $[\fraka] \in \Clk$ of $\fraka$.

We denote by $\Theta(k) \subset M_{1}(\abs{D},\chi_D)$
the space generated by all theta functions $\theta_{\fraka}$ for $[\fraka] \in \Cl_{k}$.
Note that since $D$ is a fundamental discriminant, the theta functions
$\theta_{\fraka}$ are all newforms.

Let $\psi \in \Clkd$ be a class group character and define
\begin{equation}
  \label{eq:thetapsi}
  \theta_{\psi}(\tau) = \frac{1}{w_{k}} \sum_{[\fraka] \in \Clk} \psi([\fraka]) \theta_{\fraka}(\tau).
\end{equation}
Here, $w_{k}$ is the number of roots of unity contained in $k$.

The following well known theorem (see \cite{kani-theta}) describes the space of scalar valued theta functions. Iin the clase of a prime discriminant it straightforward to derive the theorem from our results in \cref{sec:vector-valued-theta}.
\begin{theorem}\ 
  \label{thm:scalartheta}
  \begin{enumerate}
  \item If $\psi^2 = 1$, then $\theta_{\psi}$ is an Eisenstein series.
  \item If $\psi^2 \neq 1$, then $\theta_{\psi}$ is a primitive cuspidal newform.
  \item Choose a system $\calC$ of representatives of characters on $\Clk$ modulo complex conjugation. Then the set $\calB(k) = \{\theta_{\psi}\, \mid\, \psi \in \calC\}$ is an orthogonal basis for $\Theta(k)$ with respect to the
  Petersson inner product.
  \end{enumerate}
\end{theorem}

It is in fact easy to see that
\begin{equation}
  \label{eq:thetaasum}
  \theta_{\fraka}(\tau) = \frac{w_{k}}{h_{k}} \sum_{\chi \in \Clkd} \bar{\chi}([\fraka])\theta_{\chi}(\tau).  
\end{equation}

\begin{definition}
  Let $\frakA \in \Clk/\Clks$ be a genus and let $\fraka \in \frakA$.
  The Eisenstein series
  \[
    E_{\frakA}(\tau) = \frac{1}{h_{k}} \sum_{[\frakb] \in \Clk} \theta_{\fraka\frakb^2}(\tau)
  \]
  is called the (normalized) \emph{genus Eisenstein series} of $\frakA$.
\end{definition}

\begin{remark}
  The fact that $E_{\frakA}(\tau)$ is an Eisenstein series
  is again a special case of the Siegel-Weil formula (see Theorem 2.1 of \cite{bryfaltings}).
\end{remark}

\subsection{Liftings}
\begin{definition}
  We define the subspace of symmetric theta functions as
  \[
    \Theta^{\sym}(P)=\langle\, \Theta^{\sym}_{P}(\tau,h)\, \mid\,  h \Cl(K) \rangle_{\C} \subset \Theta(P),
  \]  
where $\Theta^{\sym}_{P}(\tau,h)$ is defined in \cref{eq:sym}.
\end{definition}

\begin{proposition}
  \label{prop:thetas}
  Let $\fraka \subset \OD$ be an ideal and let $(P,Q) = \left(\fraka, \frac{\N(x)}{\N(\fraka)}\right)$
  be the corresponding even quadratic lattice.
  For $h \in I_k/\ODhatt$ corresponding to the ideal class of
  $\frakb \subset \OD$, we have
  \[
    \S_{P}(\theta_{\fraka\frakb^{2}})(\tau) = \Theta_{P}^{\sym}(\tau,h).
  \]
\end{proposition}
\begin{proof}
  Note that in our case the level $N$ is equal to $\abs{D}$,
  the order of the discriminant group.
  Moreover, for $p \mid D$,
  the $\F_{p}$-rank of $A_{p}$ is equal to one. That means the ``signs''
  $\epsilon_{1},\ldots,\epsilon_{t}$, where $t$ is the number of prime divisors of $D$
  in \cref{def:epsilons} are all nonzero.
  
  We first show that the $0$-th components of $\S_{P}(\theta_{\fraka\frakb^{2}})(\tau)$
  and $\Theta_{P}^{\sym}(\tau,h)$ agree and then the claim follows from
  \cref{prop:Scheit-sqf}.
  It is clear that the $0$-th component of the function $\Theta_{P}(\tau,h)$ is equal to $\theta_{\fraka\frakb^{2}}$.
  We write
  \[
    \theta_{\fraka\frakb^{2}} = E_{\frakA}(\tau) + g_{\fraka\frakb^{2}}(\tau)
  \]
  for a cusp form $g_{\fraka\frakb^{2}}(\tau) \in S_1(\abs{D},\chi_{D})$.
  Then, it is not hard to see that in fact
  \[
    g_{\fraka\frakb^{2}}(\tau) \in S(\abs{D},\chi_{D})^{{\varepsilon_{1},\ldots,\varepsilon_{t}}}
  \]
  for ${\varepsilon_{1},\ldots,\varepsilon_{r}}$ as in \cref{def:epsilons} for the lattice $P$. 
  Indeed, we write $\chi_{P} = \chi_{D} = \prod_{i=1}^{t} \chi_{p_{i}^{*}}$, where
  \[
    \chi_{p^{*}}(n) = \legendre{p^{*}}{n} \text{ with } p^{*} = \legendre{-1}{p} p
  \]
  for a prime divisor $p$ of $D$.
  Then $\chi_{p_{i}^{*}}(n) = - \epsilon_{i}$ implies that
  the coefficient of index $n$ of $\theta_{\fraka\frakb^{2}}$ and
  of $E_{\frakA}$ vanish because the characters $\chi_{p_{i}^{*}}(n)$ are the basis of the
  genus characters.
  
  Moreover, the normalized Eisenstein series
  $E_{\frakA} \in M_{1}(\abs{D},\chi_{D})$, where $\frakA$ is the genus of $\fraka$,
  lifts to
  \[
    \S_{L}(E_{\frakA}) = \nu  E_{P}.
  \]
  \cref{prop:pet-lift} shows that the lift of an Eisenstein series is again an Eisenstein series.
  Since the Eisenstein subspace of $M_{1,P}$ is one-dimensional,
  the lift of it has to be a multiple of $E_{P}$. The correct multiple can be read off from
  \cref{eq:Fmulift} and \cref{eq:F0WN} in the proof of \cref{thm:Lift0}.

  Note that under the assumptions of the Proposition, we have $\nu = \abs{\Og(L'/L)} = 2^{t-1}$
  by \cref{prop:Scheit-sqf}.
  Thus, by \cref{thm:Lift0}, the $0$-th component of
  \[
    \Theta_{P}^{\sym}(\tau,h) - \S_{P}(\theta_{\fraka\frakb^{2}})(\tau)
  \]
  vanishes.
  Then the lemma follows from \cref{prop:Scheit-sqf} since $\Theta_P^{\sym}(\tau)$
  and $\S_{P}(\theta_{\fraka\frakb^{2}})(\tau)$ are invariant under $\Og(P'/P)$.
\end{proof}
\begin{remark}
  It follows from \cref{prop:thetas} that the space $\Theta^{\sym}(P)$ is the space
  spanned by the lifts $\calS_{P}(\theta_{\fraka\frakb^{2}})$ of the scalar valued
  theta functions $\theta_{\fraka\frakb^{2}}$ in the genus of $\fraka$.
  This establishes an isomorphism between $\Theta_{P}^{\sym}$ and the space of scalar valued
  theta functions in the genus of $\fraka$.
\end{remark}

\begin{proposition}
  Let $\calC = \Clkd$ be the group of class group characters.
  Then the set
  \[
    \calB^{\sym}(P) = \{ \Theta_{P}^{\sym}(\tau,\psi)\, \mid\, \psi \in \calC^{2} \}
  \]
  spans $\Theta^{\sym}(P) \subset M_{1,P}$.
  The elements of $\calB^{\sym}(P)$ are permuted by the action of $\Aut(\C)$.
  Moreover, $(\Theta_{P}^{\sym}(\tau,\psi),\Theta_{P}^{\sym}(\tau,\chi))=0$
  unless $\psi=\chi$ or $\psi=\overline{\chi}$.
\end{proposition}
\begin{proof}
  Using \cref{prop:thetas}, we see that $\Theta_{P}^{\sym}(\tau,\psi)$
  is in fact equal to the lift of
  \[
    w_k \sum_{\chi^{2}=\psi} \bar\chi(\fraka)\theta_{\chi}.
  \]
  Moreover, we have the relation
  $\bar\psi(\fraka)\Theta_{P}(\tau,{\psi}) = \overline{\Theta_{P}(\tau,{\psi})} = \Theta_{P}(\tau,{\bar\psi})$.
  The result follows from \cref{thm:thetaPbasis}.

\end{proof}

\begin{corollary}
  Let $\bar{\calC}^{2}$ be a set of representatives of
  $\calC^{2}$ modulo the relation $\chi \mapsto \bar\chi$.
  Then the set
  \[
     \{ \Theta_{P}^{\sym}(\tau,\psi)\, \mid\, \psi \in \bar{\calC}^{2} \}
  \]
  is an orthogonal basis of $\Theta^{\sym}(P)$.
\end{corollary}

\begin{corollary}
  \label{cor:scalar-theta-pet}
  Let $\psi \in \Clkd, \psi \neq 1$.
  We have 
  \[
     w_{k}^{2} \sum_{\chi^{2} = \psi} (1 + \bar{\chi}^{2}(\fraka))(\theta_{\chi}(\tau),\theta_{\chi}(\tau))
     = (\Theta_{P}^{\sym}(\tau,\psi),\Theta_{P}^{\sym}(\tau,\psi)).
  \]
\end{corollary}
\begin{proof}
  We expand the right hand side and obtain
  \begin{align*}
    (\Theta_{P}^{\sym}(\tau,\psi),\Theta_{P}^{\sym}(\tau,\psi))
      &= w_{k}^{2}(\sum_{\chi^{2} = \psi} \bar{\chi}(\fraka) \theta_{\chi}(\tau), \sum_{\lambda^{2} = \psi} \bar\lambda(\fraka) \theta_{\lambda}(\tau))\\
      & = w_k^2\sum_{\chi^{2} = \psi} (1 + \chi^{2}(\fraka)) (\theta_{\chi}(\tau), \theta_{\chi}(\tau)).\qedhere
  \end{align*}
\end{proof}

We obtain the following well-known formula for the Petersson norm of the scalar valued cusp forms associated with theta functions of positive definite binary quadratic forms.
\begin{corollary}
  \label{cor:scalar-theta-pet-expl}
  Suppose that $D=-p$ is a prime discriminant.
  Let $\chi \in \Clkd$ with $\chi \neq 1$ and write again $\tau(\fraka) = u(\fraka) + i v(\fraka)$.
  Then we have
  \[
    (\theta_{\chi}(\tau),\theta_{\chi}(\tau))
      = -\frac{4 h_{k}}{w_{k}^{2}} \sum_{\fraka \in \Clk}\chi^{2}(\fraka) \log\abs{v(\fraka)^{1/2}\eta^{2}(\tau(\fraka))}.
  \]
\end{corollary}
\begin{proof}
  We use \cref{cor:scalar-theta-pet} with $P \cong \OD$ for $D=-p$
  together with \cref{thm:thetaPbasis}. Moreover, we have to use the fact that for
  prime discriminants, the class number is odd and therefore, the sum in \cref{cor:scalar-theta-pet}
  reduces to a single term.
\end{proof}

\bibliographystyle{alpha}
\def\cprime{$'$}

\end{document}